\documentclass[a4paper,10pt]{article}
\usepackage[utf8x]{inputenc}
\pdfoutput=1

\usepackage{times}
\usepackage{geometry}
\usepackage[T1]{fontenc}
\usepackage{color}
\usepackage{amsmath}
\usepackage{amssymb}
\usepackage{array}
\usepackage{amsthm}
\usepackage{graphicx}
\usepackage{hyperref}
\usepackage{setspace}
\usepackage{stmaryrd}

\theoremstyle{plain}
\newtheorem{thm}{Theorem}[section]
\newtheorem{prop}[thm]{Proposition}
\newtheorem{cor}[thm]{Corollary}
\newtheorem{lem}[thm]{Lemma}

\theoremstyle{definition}
\newtheorem{defn}[thm]{Definition}

\theoremstyle{remark}
\newtheorem{rem}[thm]{Remark}

\theoremstyle{plain}

\newcommand{\R}{\mathbb{R}}

\newcommand{\N}{\mathbb{N}}

\newcommand{\CP}{\mathbb{C}P}

\newcommand{\ind}{\mathrm{ind}}

\newcommand{\dimn}{\mathrm{dim}}

\newcommand{\identity}{\mathrm{id}}
\newcommand{\scal}{\mathrm{scal}}
\newcommand{\ric}{\mathrm{Ric}}
\DeclareMathOperator{\trace}{tr}
\newcommand{\kernel}{\mathrm{ker}}

\newcommand{\mult}{\mathrm{mult}}
\newcommand{\dv}{\text{ }dV}

\newcommand{\Diff}{\mathrm{Diff}}

\newcommand{\spectrum}{\mathrm{spec}}

\newcommand{\Parallel}{\mathrm{par}}

\newcommand{\Biggmid}{\hspace{1mm}\Bigg\vert\hspace{1mm}}

\renewcommand{\title}[1]{{\bfseries #1}\par}
\renewcommand{\author}[1]{\medskip{#1}\par\smallskip}
\newcommand{\affiliation}[1]{{\itshape #1}\par}
\newcommand{\email}[1]{E-mail:~\texttt{#1}\par}

\numberwithin{equation}{section}

\begin{document}
\begin{center}
\title{\LARGE On infinitesimal Einstein deformations}
\vspace{3mm}
\author{\Large Klaus Kröncke}
\vspace{3mm}
\affiliation{Universität Regesnburg, Fakultät für Mathematik\\Universitätsstraße 31\\93053 Regensburg, Germany}
\vspace{3mm} 
\affiliation{Universität Potsdam, Institut für Mathematik\\Am Neuen Palais 10\\14469 Potsdam, Germany} 
\email{klaus.kroencke@mathematik.uni-regensburg.de} 
\end{center}
\vspace{2mm}
\begin{abstract}We study infinitesimal Einstein deformations on compact flat manifolds and on product manifolds. Moreover, we prove refinements of results by Koiso and Bourguignon which yield
obstructions on the existence of infinitesimal Einstein deformations under certain curvature conditions.
\end{abstract}
 \section{Introduction}
 Let $M^n$ be a compact manifold of dimension $n\geq3$ and let $\mathcal{M}$ be the set of smooth Riemannian metrics on it.
%  For any $c>0$, let $\mathcal{M}_c\subset\mathcal{M}$ be the subset of smooth Riemannian metrics of volume $c$. 
Given an Einstein metric $g$, one may ask whether $g$ is isolated in the set of Einstein structures, i.e.\ any other Einstein metric in a small neighbourhood in $\mathcal{M}$ is homothetic to $g$.

To study this question, we consider infinitesimal Einstein deformations, that are symmetric $2$-tensors $h$ which are trace-free and divergence-free and satisfy the linearized Einstein equations
\begin{align*}
\Delta_Eh:= \nabla^*\nabla h-2\mathring{R}h=0.
\end{align*}
Trace-free and divergence-free symmetric $2$-tensors are also called $TT$-tensors.
By ellipticity of the involved operator, the space of infinitesimal Einstein deformations is finite dimensional since $M$ is compact. If $g$ has no infinitesimal Einstein deformations, it is isolated in the space of Einstein structures. 
The converse is not true: The product metric on $S^2\times \CP^{2n}$ is isolated although it has infinitesimal Einstein deformations \cite{Koi82}.

Moreover, as is well-known, Einstein metrics of volume $c$ are critical points of the Einstein-Hilbert action
\begin{align*}S\colon\mathcal{M}_c\ni g&\mapsto \int_M\scal_g\dv_g
\end{align*}
\cite{Hil15}. Here, $\mathcal{M}_c$ is the set of Einstein metrics of volume $c$.
 Einstein metrics are always saddle points of the Einstein-Hilbert action but there is a notion of stability of Einstein manifolds which is as follows: 
We say that an Einstein manifold is stable, if $S''(h)\leq0$ for all $TT$-tensors. We call it strictly stable if $S''(h)<0$ for all nonzero $TT$-tensors.
An Einstein manifold is strictly stable if and only if it is stable
 and does not admit infinitesimal Einstein deformations.
This stability problem has been studied by Koiso \cite{Koi78,Koi80,Koi83}, Dai, Wang and Wei \cite{DWW05,DWW07} and in a recent paper by the author \cite{Kro14}.

In this work, we study infinitesimal Einstein deformations on certain classes of manifolds.
  Throughout, any manifold $M^n$ is compact and $n\geq3$ unless the contrary is explicitly asserted. This work is organized as follows:
In section \ref{bieberbach}, we consider compact flat manifolds. We compute the dimension of infinitesimal Einstein deformations in terms of the holonomy of the manifold:
 \begin{thm}\label{bieberbachdeformation}Let $(M=\R^n/G,g)$ be a Bieberbach manifold\index{Bieberbach!manifold} and let $\rho$ be the canonical representation of the holonomy of $G$ on $\R^n$. Let
 \begin{align*}\rho\cong(\rho_{1})^{i_1}\oplus\ldots \oplus (\rho_l)^{i_l}
 \end{align*}
 be an irreducible decomposition of $\rho$. Then the dimension of the space of infinitesimal Einstein deformations is equal to
 \begin{align*}\dimn(\kernel(\Delta_{E}|_{TT})=\sum_{j=1}^l \frac{i_j(i_j+1)}{2}-1.
\end{align*}
 \end{thm}
Here, $TT$ denotes the space of $TT$-tensors.
In section \ref{productmanifolds}, we consider products of Einstein spaces and we compute the kernel and the coindex of $S''$ restricted to $TT$-tensors on products of Einstein spaces.
As a result of our discussion, we get
\begin{thm}\label{existenceIED}Let $(M,g_1)$ be an Einstein manifold with positive Einstein constant $\mu$ and suppose, $2\mu\in\spectrum(\Delta)$. Then for any other Einstein manifold $(N,g_2)$ with the same Einstein constant, the product manifold $(M\times N, g_1+g_2)$
admits infinitesimal Einstein deformations.
\end{thm}
The dimension of the space of infinitesimal Einstein deformations is bounded from below by the multiplicity of the eigenvalue $2\mu$.
By a result of Matsushima (\cite{Mat72}, see also \cite[Theorem 11.52]{Bes08}),
a K\"ahler-Einstein metric with Einstein constant $\mu$ admits a holomorphic vector field if and only if $2\mu$ is contained in the spectrum of the Laplacian.
Therefore we obtain
\begin{cor}
 Let $(M,g_1)$ be a positive K\"ahler-Einstein manifold which admits a holomorphic vector field. Then for any other Einstein manifold $(N,g_2)$ with the same Einstein constant, the product manifold $(M\times N, g_1+g_2)$
admits infinitesimal Einstein deformations.
\end{cor}
This allows us to generate large families of Einstein manifolds which have infinitesimal Einstein deformations. In fact, all known K\"ahler-Einstein manifolds with $c_1>0$ admit holomorphic vector fields \cite[Remark 12.101]{Bes08}.
In section  \ref{stabilitysectional}, we refine the following stability criterions which are well-known from the literature:
\begin{cor}[Bourguignon, unpublished]\label{pinching}           
 Let $(M,g)$ be an Einstein manifold such that the sectional curvature lies in the interval $(\frac{n-2}{3n},1]$. Then $(M,g)$ is strictly stable.
\end{cor}
\begin{cor}[{{\cite[Proposition 3.4]{Koi78}}}]\label{stabilitywhenK<0}
 Let $(M,g)$ be an Einstein manifold with sectional curvature $K<0$. Then $(M,g)$ is strictly stable.
\end{cor}
We consider the boundary cases of these corollaries and observe that the existence of infinitesimal Einstein deformation imposes strong conditions on the manifold: It is even-dimensional and the tangent bundle splits in two subbundles
of the same dimension which admit certain properties (Propositions \ref{splittingtheorem1} and \ref{splittingtheorem2}).

\vspace{3mm}

\textbf{Acknowledgement.} This article is based on a part of the authors PhD-thesis. The author would like to thank his advisor Christian B\"ar for helpful discussions. Moreover, the 
author thanks the Max-Planck Institute for Gravitational Physics for financial support.
\section{Preliminaries}
Let us first fix some notation and conventions. We define the Laplace-Beltrami operator acting on functions by $\Delta=-\trace \nabla^2$. For the Riemann curvature tensor, we use the sign convention such that
 $R_{X,Y}Z=\nabla^2_{X,Y}Z-\nabla^2_{Y,X}Z$. Given a fixed metric, we equip the bundle of $(r,s)$-tensor fields (and any subbundle) with the natural scalar product induced by the metric.
By $S^pM$, we denote the bundle of symmetric $(0,p)$-tensors.
 Let $\left\{e_1,\ldots,e_n\right\}$ be a local orthonormal frame. The divergence is the map $\delta:\Gamma(S^pM)\to\Gamma(S^{p-1}M)$, defined by
\begin{align*}\delta T(X_1,\ldots,X_{p-1})=-\sum_{i=1}^n\nabla_{e_i}T(e_i,X_1,\ldots,X_{p-1})
\end{align*}
and its adjoint $\delta^*\colon\Gamma(S^{p-1}M)\to \Gamma(S^pM)$ with respect to the natural $L^2$-scalar product is given by
\begin{align*}\delta^*T(X_1,\ldots,X_p)=\frac{1}{p}\sum_{i=0}^{p-1}\nabla_{X_{1+i}}T(X_{2+i},\ldots,X_{p+i}),
\end{align*}
where the sums $1+i,\ldots,p+i$ are taken modulo $p$.

 The second variation of $S$ at Einstein metrics was studied in \cite{Koi79}. For details, see also \cite[Chapter 4]{Bes08}.
 Any compact Einstein metric except the standard sphere admits the decomposition
 \begin{align}\label{decomp}T_g\mathcal{M}=\Gamma(S^2M)=C^{\infty}(M)\cdot g\oplus \delta_g^{*}(\Omega^1(M))\oplus \trace_g^{-1}(0)\cap\delta_g^{-1}(0)
 \end{align}
 and these factors are all infinite-dimensional. This splitting is orthogonal with respect to $S''$, so the second variation can be studied separately on each of these factors.

The first factor of \eqref{decomp} is the tangent space of the conformal class of $g$. It is known that $S''$ is nonnegative on volume-preserving conformal deformations.
This is due to the fact that any Einstein metric is Yamabe \cite[p.\ 329]{LeB99}, i.e.\ it is a minimizer of the (volume-normalized) total scalar curvature in its conformal class.
The second factor is the tangent space
of the orbit of the diffeomorphism group acting on $g$. By diffeomorphism invariance, $S''$ vanishes on this factor.

The tensors in the third factor are also often called transverse traceless or $TT$.
From now on, we abbreviate $TT_g=\trace_g^{-1}(0)\cap\delta_g^{-1}(0)$.
The second variation of $S$ on $TT$-tensors is given by
\begin{align*}S''(h)=-\frac{1}{2}\int_M\langle h, \nabla^*\nabla h -2\mathring{R}h\rangle\dv.
\end{align*}
Here, $\mathring{R}$ is the action of the curvature tensor on symmetric $(0,2)$-tensors, given by 
\begin{align*}\mathring{R}h(X,Y)=\sum_{i=1}^nh(R_{e_i,X}Y,e_i).
 \end{align*}
\begin{defn}We call the operator $\Delta_E:\Gamma(S^2M)\to\Gamma(S^2M)$, $\Delta_Eh=\nabla^*\nabla h -2\mathring{R}h$ the Einstein operator.
 \end{defn}
This is a self-adjoint elliptic operator and by compactness of $M$, it has a discrete spectrum. The Einstein operator preserves all components of the splitting \eqref{decomp}.
\begin{defn}\label{stability}We call a compact Einstein manifold $(M,g)$ stable, if the Einstein operator is nonnegative on $TT$-tensors and strictly stable, if it is positive on $TT$-tensors.
 We call $(M,g)$ unstable, if the Einstein operator admits negative eigenvalues on $TT$. Furthermore, elements in $\kernel(\Delta_E|_{TT})$ are called infinitesimal Einstein deformations.
\end{defn}
\begin{rem}
If $g_t$ is a nontrivial curve of Einstein metrics through $g=g_0$ orthogonal to $\R\cdot(g\cdot \Diff(M))$, then $\dot{g}_0$ is an infinitesimal Einstein deformation.
Evidently, an Einstein manifold is isolated (or rigid) in the space of Einstein structures if $\Delta_E|_{TT}$ has trivial kernel.
\end{rem}
\begin{defn}An infinitesimal Einstein deformation $h$ is called integrable if there exists a curve of Einstein metrics tangent to $h$.
 \end{defn}

 \section{Einstein deformations of Bieberbach manifolds}\label{bieberbach}
 Bieberbach manifolds are flat\index{flat} connected compact manifolds. It is well known that any Bieberbach manifold\index{Bieberbach!manifold}
 is isometric\index{isometric} to $\R^n/G$, where $G$ is a suitable subgroup\index{subgroup} of the Euclidean motions\index{Euclidean motions} $E(n)=O(n)\ltimes \R^n$.
 We call such groups Bieberbach groups\index{Bieberbach!group}.
 For every element $g\in E(n)$, there exist unique $A\in O(n)$ and $a\in\R^n$ such that
 $gx=Ax+a$ for all $x\in\R^n$, and we write $g=(A,a)$.
 There exist homomorphisms $r\colon E(n)\to O(n)$ and $t\colon\R^n\to E(n)$, defined by $r(A,a)=A$ and $t(a)=(1,a)$.
 Let $G$ be a Bieberbach group. 
The subgroup $r(G)\subset O(n)$ is called the holonomy\index{holonomy} of $G$ since its natural representation on $\R^n$ is equivalent to the holonomy representation of $\R^n/G$ (see e.g.\ \cite[pp.\ 50-52]{Cha86}).
 
 We call two Bieberbach manifolds $M_1$ and $M_2$ affinely equivalent\index{affinely equivalent} if there exists a diffeomorphism $F:M_1\to M_2$
 whose lift to the universal coverings\index{universal covering} $\pi_1\colon\R^n\to M_1$, $\pi_2\colon\R^n\to M_2$ is an affine map \index{affine map}$\alpha\colon\R^n\to\R^n$
 such that
\begin{align*}\pi_2\circ\alpha=F\circ\pi_1.
\end{align*}

 If $M_1$ and $M_2$ are affinely equivalent, the corresponding Bieberbach groups $G_1$ and $G_2$ are isomorphic via $\varphi:G_1\to G_2$,
 $\varphi(g)=\alpha g\alpha^{-1}$. Conversely, if two Bieberbach groups $G_1$ and $G_2$ are isomorphic, there exists an affine map\index{affine map}
 $\alpha$ such that the isomorphism\index{isomorphism} is given by $g\mapsto \alpha g\alpha^{-1}$ (see \cite[Theorem 3.2.2]{Wol11}). The map $\alpha$ descends to a diffeomorphism
 $F:M_1\to M_2$ and $M_1$ and $M_2$ are affinely equivalent via $F$.

 Now we want to determine whether a Bieberbach manifold has infinitesimal Einstein deformations.
 Any Bieberbach manifold is stable since
 \begin{align*}(\Delta_Eh,h)_{L^2}=(\nabla^*\nabla h,h)_{L^2}=\left\|\nabla h\right\|_{L^2}^2\geq0.
 \end{align*}
 Furthermore, any infinitesimal Einstein deformation is parallel. 
 In the following, we will compute the dimension of the kernel of $\Delta_E=\nabla^*\nabla$ in terms of the holonomy. 
The following lemma is a consequence of the holonomy principle.\index{holonomy!principle}
\begin{lem}[\cite{Die13}, Proposition 4.2]\label{holonomy}
 Let $(M,g)$ be a connected Riemannian manifold and let $h$ be a symmetric $(0,2)$-tensor field. Let $p\in M$ and let $T_pM=(E_1)_p\oplus\ldots\oplus (E_k)_p$ be a decomposition into irreducible $Hol_p(M,g)$ representations and let
\begin{align*}
TM=E_1\oplus\ldots\oplus E_k
 \end{align*}
be the decomposition of the tangent bundle obtained by parallel transport of the $(E_i)_p$. Then $\nabla h=0$ if and only if $h=\sum_{i=1}^k\lambda_i g_i$ where $\lambda_i\in\R$ and $g_i$ is the metric restricted to $E_i$.
\end{lem}
\begin{proof}Consider $h$ as an endomorphism on $TM$ and suppose that $\nabla h=0$. By the holonomy principle, $h_p$ commutes with the holonomy representation, i.e.\ $h_p\circ \rho(g)=\rho(g)\circ h_p$ for all $g\in Hol_p(M,g)$.
 By Schur's lemma, $h_p=\sum_{i=1}^k\lambda_k(\mathrm{pr}_i)_p$, where $(\mathrm{pr}_i)_p:T_pM\to(E_i)_p$ is the projection map. Let $\mathrm{pr}_i:TM\to E_i$ be the global projection map.
 It follows that $h=\sum_{i=1}^k\lambda_i\mathrm{pr}_i$, since we obtain $\mathrm{pr}_i$ from $(\mathrm{pr}_i)_p$ via parallel transport.
The converse is clear.
\end{proof}
\begin{cor}\label{holonomy2}
 Let $(M,g)$ be a connected Riemannian manifold. Then there exists a traceless symmetric $(0,2)$-tensor field with $\nabla h=0$ if and only if the holonomy of $(M,g)$ is reducible.
\end{cor}
\begin{proof}This follows from Lemma \ref{holonomy} since any traceless symmetric $(0,2)$-tensor field admits at least two distinct eigenvalues.
 \end{proof}

 \begin{cor}A Bieberbach manifold\index{Bieberbach!manifold} $M=\R^n/G$ is strictly stable if and only if the subgroup $r(G)\subset O(n)$ acts irreducibly on $\R^n$.
 \end{cor}
 \begin{proof}Since the canonical representation of $r(G)$ on $\R^n$ is equivalent to the holonomy\index{holonomy} representation of $M$
and any infinitesimal Einstein deformation is parallel, the assertion is immediate from Corollary \ref{holonomy2}.
 \end{proof}

 For the moment, let $(M,g)$ be an arbitrary Riemannian manifold.
 We compute the dimension of the space of parallel symmetric $(0,2)$-tensors on $(M,g)$ in terms of the holonomy.
 Let $TM=E_1\oplus\ldots\oplus E_k$ be a parallel orthogonal splitting of the tangent bundle in irreducible\index{irreducible} components.
 Then a parallel splitting\index{parallel!splitting} of the bundle of symmetric $(0,2)$-tensors is given by
 \begin{align}\label{symmetrictensorsplitting}T^*M\odot T^*M=\bigoplus_{i,j=1}^k E_i^*\odot E_j^*=\bigoplus_{i=1}^k \odot^2 E_i^*\oplus\bigoplus_{i<j}^k E_i^*\odot E_j^*.
\end{align}
 Here, $E_i^*$ is the image of $E_i$ under the musical isomorphism\index{musical isomorphism} and $\odot$ denotes the symmetric tensor product\index{symmetric tensor product}\index{$\odot$, symmetric tensor product}.
 We now search the parallel sections in each of these summands. First suppose that $h\in \Gamma(\odot^2 E_i)$ is parallel.
 Considered as an endomorphism\index{endomorphism} on $TM$, it induces a parallel endomorphism $h:E_i\to E_i$.
 By the proof of Lemma \ref{holonomy}, its eigensections\index{eigensection} form a splitting of the bundle $E_i$. Since $E_i$ is irreducible\index{irreducible},
 $h=\lambda g_i$ where $\lambda\in\R$ and $g_i$ is the metric restricted to $E_i$.

Now we consider the second component of the splitting (\ref{symmetrictensorsplitting}).
 Sections of $E^*_i\odot E^*_j$, considered as endomorphisms on $TM$, are sections of $\mathrm{End}(E_i\oplus E_j)$\index{$\mathrm{End}$, endomorphism bundle} which are of the form
 \begin{align*}h=\begin{pmatrix}0 & A^*\\
              A & 0
           \end{pmatrix},
 \end{align*}
 where $A\in\Gamma(\mathrm{End}(E_i,E_j))$ and $A^*$ is its adjoint\index{adjoint map}. If $h$ is parallel, $A$ is also parallel.
% Therefore, $\kernel(A)$ and $\mathrm{im}(A)$ are both parallel subbundles\index{subbundle} of $E_i,E_j$, respectively.
% Since $E_i,E_j$ are irreducible, this shows that $A$ is an isomorphism if it is nonzero. 
 Fix a point $p$ and consider a linear map $A_p:(E_i)_p\to (E_j)_p$.\index{$(E_i)_p$, fiber of the bundle $E_i$ at $p$}
 By the holonomy principle, $A_p$ can be extended to a parallel endomorphism $A:E_i\to E_j$ if and only if $A_p$ commutes with 
 %the holonomy representation $\rho(Hol_p(M,g))\subset O(T_pM,g_p)$.
  the restricted holonomy representatios $\rho(Hol_p(M,g))|_{E_i}$
 and $\rho(Hol_p(M,g))|_{E_j}$. 
 Since these representations are irreducible,
 Schur's lemma implies that $A_p$ is either zero or an isomorphism. In the latter case, these representations are equivalent via $A_p$ and any other equivalence is a multiple of $A_p$. 
% 
% the space of linear maps $L:(E_i)_p\to(E_j)_p$ commuting with them is $1$-dimensional.
% This follows easily from a lemma from representation theory (see e.g.\ \cite[p.~27]{NS82}).

 In summary, we have shown that the dimension of the space of parallel sections in $E^*_i\odot E^*_j$ equals $1$ if the holonomy representations\index{holonomy!representation} restricted to
 $E_i$ and $E_j$ are equivalent and zero otherwise. Summing over all components of the splitting \eqref{symmetrictensorsplitting}, we obtain
 \begin{prop}\label{paralleltensors}Let $(M,g)$ be a Riemannian manifold and let $Hol(M,g)$ be its holonomy representation. Let 
 \begin{align*}Hol(M,g)\cong(\rho_{1})^{i_1}\oplus\ldots \oplus (\rho_l)^{i_l}
 \end{align*}
 be an irreducible decomposition of $Hol(M,g)$. Then the dimension of parallel symmetric $(0,2)$-tensors is equal to
 \begin{align*}\dimn(\Parallel(S^2M))=\sum_{j=1}^l \frac{i_j(i_j+1)}{2}.
\end{align*}
 \end{prop}
 Let us now go back to the special case of a Bieberbach manifold $(\R^n/G,g)$ and recall that infinitesimal Einstein deformations are precisely the traceless parallel symmetric $(0,2)$-tensors.
 Using the fact that the canonical representation $r:G\to O(n)$ is equivalent to the holonomy representation of $M$, we obtain Theorem \ref{bieberbachdeformation}.

  \begin{rem}Any infinitesimal Einstein deformations on a Bieberbach manifold if integrable since $g+th$ is a curve of flat metrics, if $g$ is flat and $h$ is parallel.
 \end{rem}
 Recall that two Bieberbach manifolds $M_1$ and $M_2$ are called affinely equivalent\index{affinely equivalent} if there exists a diffeomorphism $F:M_1\to M_2$ whose lift to the universal coverings\index{universal covering} $\pi_1\colon\R^n\to M_1$, $\pi_2\colon\R^n\to M_2$
 is an affine map $\alpha\in GL(n)\ltimes \R^n$ such that $F\circ\pi_1=\pi_2\circ \alpha$.
 Since $\pi_1,\pi_2$ are local isometries\index{local isometry} and $\alpha$ is affine, the map $F$ is parallel.
 Therefore, the induced map $F_*\colon\Gamma(S^2M_1)\to \Gamma(S^2M_2)$ maps parallel tensor fields\index{parallel!tensor field} on $M_1$ isomorphically to parallel tensor fields on $M_2$.
 It follows that the dimension of infinitesimal Einstein deformations only depends on the affine equivalence class\index{affine equivalence class} of $M$.

 For any $n\in\N$ the number of affine equivalence classes of $n$-dimensional Bieberbach manifolds\index{Bieberbach!manifold} is finite \cite{Bie12}.
In dimension $3$, a classification of all Bieberbach manifolds up to affine equivalence is known. In fact, there exist $10$ Bieberbach $3$-manifolds where
 six of them are orientable\index{orientable} and the others are non-orientable\index{non-orientable}. We describe the corresponding Bieberbach groups\index{Bieberbach!group} in the following. Moreover, 
 we will compute the dimension of infinitesimal Einstein deformations explicitly. Let $\left\{e_1,e_2,e_3\right\}$ be the standard basis\index{standard basis} of
 $\R^3$, let $R(\varphi)$ be the rotation matrix\index{rotation matrix} of rotation of $\R^3$ about the $e_1$-axis through $\varphi$ and let $E$ be the reflection
 matrix\index{reflection matrix} at the $e_1$-$e_2$-plane, i.e.\
 \begin{align*}e_1=\begin{pmatrix}1\\0\\0\end{pmatrix},\qquad e_2=\begin{pmatrix}0\\1\\0\end{pmatrix},\qquad e_3=\begin{pmatrix}0\\0\\1\end{pmatrix},
 \\ R(\varphi)=\begin{pmatrix}1 & 0 & 0 \\ 
 0 & \cos(\varphi) & -\sin(\varphi) \\ 
 0 & \sin(\varphi) & \cos(\varphi)\end{pmatrix},\qquad
 E=\begin{pmatrix} 1 & 0 & 0 \\
                   0 & 1 & 0 \\
                   0 & 0 & -1
 \end{pmatrix}.
 \end{align*}
 Let furthermore $t_i=(I,e_i)$, $i\in\left\{1,2,3\right\}$ and $I$ be the identity map. Then the Bieberbach groups\index{Bieberbach!group} can be described as follows (see  e.g.\ \cite[Lemma 2.1]{Kan06}):
 \begin{center}
 \begin{tabular}{|l | l |}
\hline
   & generators\index{generator} of $G_i$ \\
 \hline
 $G_1$ & $t_1,t_2,t_3$ \\
  $G_2$ & $t_1,t_2,t_3$ and $\alpha=(R_\pi,\frac{1}{2}e_1)$ \\
  $G_3$ & $t_1,s_1=(I,R_\frac{2\pi}{3}e_2),s_2=(I,(R_\frac{4\pi}{2}e_2))$ and $\alpha=(R_{\frac{2\pi}{3}},\frac{1}{3}e_1)$ \\
  $G_4$ & $t_1,t_2,t_3$ and $\alpha=(R_\frac{\pi}{2},\frac{1}{4}e_1)$ \\
  $G_5$ & $t_1,s_1=(I,R_\frac{\pi}{3}e_2),s_2=(R(\frac{2\pi}{3})e_2,I)$ and $\alpha=(R_\frac{\pi}{3},\frac{1}{6}e_1)$ \\
  $G_6$ & $t_1,t_2,t_3,\alpha=(R_\pi,\frac{1}{2}e_1)$, \\
 & $\beta=(-E\cdot R_\pi,\frac{1}{2}(e_2+e_3))$ and $\gamma=(-E,\frac{1}{2}(e_1+e_2+e_3))$ \\
  $G_7$ & $t_1,t_2,t_3$ and $\alpha=(E,\frac{1}{2}e_1)$ \\
  $G_8$ & $t_1,t_2,s=(I,\frac{1}{2}(e_1+e_2)+e_3)$ and $\alpha=(E,\frac{1}{2}e_1)$ \\
  $G_9$ & $t_1,t_2,t_3,\alpha=(R_\pi,\frac{1}{2}e_1)$ and $\beta=(E,\frac{1}{2}e_2)$ \\
  $G_{10}$ & $t_1,t_2,t_3,\alpha=(R_\pi,\frac{1}{2}e_1)$ and $\beta=(E,\frac{1}{2}(e_2+e_3))$\\
\hline 
\end{tabular}
 \end{center}
 The manifolds $M/G_i$ are orientable\index{orientable} if $1\leq i\leq 6$ and non-orientable if $7\leq i\leq 10$.
 Now we extract the generators of the holonomy\index{holonomy} and use Theorem \ref{bieberbachdeformation} to compute the dimension of $\kernel(\Delta_E|_{TT})$:
 \begin{center}
  \begin{tabular}{|l | l | l|}
 \hline
   & generators of $r(G_i)$ & $\dimn(\kernel\Delta_E|_{TT})$\\
 \hline
 $G_1$ & I & $5$\\
 $G_2$ & $R_\pi$ & $3$ \\
 $G_3$ & $R_\frac{2\pi}{3}$ & $1$ \\
 $G_4$ & $R_\frac{\pi}{2}$ & $1$ \\
 $G_5$ & $R_\frac{\pi}{3}$ & $1$ \\
 $G_6$ & $\left\{R_\pi,-E\cdot R_\pi,-E\right\}$ & $2$ \\
 $G_7$ & $E$ & $3$ \\
 $G_8$ & $E$ & $3$ \\
 $G_9$ & $\left\{R_\pi,E\right\}$ & $2$\\
 $G_{10}$ & $\left\{R_\pi,E\right\}$ & $2$\\
 \hline
 \end{tabular}
 \end{center}

 This table in particular shows that each three-dimensional Bieberbach manifold\index{Bieberbach!manifold} has infinitesimal Einstein deformations
 and hence, it is also deformable as an Einstein space by our remark above. In fact, the moduli space of Einstein structures\index{moduli space!of Einstein structures} on these manifolds concides with the moduli space of flat structures
\index{moduli space!of flat structures}.
 An explicit desciption of these moduli spaces is given in \cite[Theorem 4.5]{Kan06}.
\begin{rem}
 It seems possible but it is not known if there are Bieberbach manifolds which are isolated as Einstein spaces.
\end{rem}
 \section{The Einstein operator on product manifolds}\label{productmanifolds}
 Let $(M,g_1)$ and $(N,g_2)$ be Einstein manifolds and consider the product\index{product} manifold $(M\times N,g_1+ g_2)$.
 It is Einstein if and only if the components have the same Einstein constant $\mu$. In this case, the Einstein constant
 of the product is also $\mu$. We want to determine if a product Einstein space is stable or not. 
 This was worked out in \cite{AMo11} in the case, where the Einstein constant is negative. We study the general case. 
 
In the following, we often lift tensors on the factors $M,N$ to tensors on $M\times N$ by pulling back along the projection maps\index{projection}. In order to avoid notational complications, we drop the explicit reference to the projections throughout the section.

At first, we consider the spectrum of the Einstein operator on the product space.
 \begin{prop}[\cite{AMo11}]\label{productspectrum}Let $\Delta_E^{M\times N}$ be the Einstein operator with respect to the product metric acting on $\Gamma(S^2(M\times N))$.
 Then the spectrum\index{spectrum} of $\Delta_E^{M\times N}$ is given by
 \begin{align*}\spectrum(\Delta_E^{M\times N})=(\spectrum(\Delta_E^M)+\spectrum(\Delta_0^N))&\cup(\spectrum(\Delta_E^N)+\spectrum(\Delta_0^M))\\
               &\cup (\spectrum(\Delta_1^M)+\spectrum(\Delta_1^N)).
 \end{align*}
 Here, $\Delta_0^M$, $\Delta_0^N$, $\Delta_1^M$, $\Delta_1^N$\index{$\Delta_1$, connection Laplacian on $\Omega^1(M)$} denote the connection Laplacians\index{Laplacian!connection} on functions and $1$-forms\index{1@$1$-form} with respect to the
 metrics on $M$ and $N$, respectively.
 \end{prop}
 \begin{proof}Let $\left\{v_i\right\}$, $\left\{\alpha_i\right\}$, $\left\{h_i\right\}$ be complete orthonormal systems of symmetric $p$-eigentensors
 ($p=0,1,2$) of the operators $\Delta_0^M$, $\Delta_1^M$, $\Delta_E^M$, respectively. Let $\lambda_i^{(0)},\lambda_i^{(1)},\lambda_i^{(2)}$ be the corresponding eigenvalues.
 Let $\left\{w_i\right\}$, $\left\{\beta_i\right\}$, $\left\{k_i\right\}$ be complete orthonormal systems of symmetric $(0,p)$-eigentensors ($p=0,1,2$)
 of the operators $\Delta_0^N$, $\Delta_1^N$, $\Delta_E^N$, respectively. Let $\kappa_i^{(0)},\kappa_i^{(1)},\kappa_i^{(2)}$ be their eigenvalues.
 By \cite[Lemma 3.1]{AMo11}, the tensor products $v_ik_j$, $w_ih_j$, $\alpha_i\odot\beta_j$ form a complete orthonormal system in
 $\Gamma(S^2(M\times N))$. Straightforward calculations show that 
 \begin{align*}\Delta_E^{M\times N}(v_ik_j)&=(\lambda^{(0)}_i+\kappa^{(2)}_j)v_ik_j,\\
 \Delta_E^{M\times N}(\alpha_i\odot\beta_j)&=(\lambda^{(1)}_i+\kappa^{(1)}_j)\alpha_i\odot\beta_j,\\
 \Delta_E^{M\times N}(w_ih_j)&=(\kappa^{(0)}_i+\lambda^{(2)}_j)w_ih_j,
 \end{align*}
 from which the assertion follows.
 \end{proof}
Another operator closely related to the Einstein operator is the Lichnerowicz Laplacian acting on $\Gamma(S^2M)$, defined by
\begin{align}\label{LLdef}
 \Delta_L h=\nabla^*\nabla h+\ric\circ h+h\circ\ric-2\mathring{R}h.
\end{align}
It satisfies some useful properties:
\begin{lem}\label{LL}Let $(M,g)$ be a Riemannian manifold. Then
\begin{align}\label{LL1}\Delta_L(f\cdot g)=&(\Delta f)\cdot g,\\
             \label{LL2} \trace \left(\Delta_L h\right)=&\Delta(\trace h)
\end{align}
for all $f\in C^{\infty}(M),$ $h\in \Gamma(S^2M)$.
Moreover, if $\ric$ is parallel,
\begin{align}\label{LL3}\Delta_L(\delta^*\alpha)&=\delta^*( \Delta_H\alpha),\\
              \label{LL4}\delta(\Delta_Lh)&=\Delta_H(\delta h),\\
              \label{LL5}\Delta_L(\nabla^2 f)&=\nabla^2(\Delta f)
\end{align}
for all $f\in C^{\infty}(M),$ $\alpha\in\Omega^1(M)$, $h\in \Gamma(S^2M)$.\index{$\Delta_H$, Hodge Laplacian}
Here, $\Delta_H=\nabla^*\nabla-\ric$ is the Hodge Laplacian\index{Laplacian!Hodge} on $1$-forms\index{1@$1$-form}.
\end{lem}
 \begin{proof}Formula \eqref{LL1} follows from an easy calculation. Formula \eqref{LL5} follows from \eqref{LL3} and the well-known formula $\Delta_H(\nabla f)=\nabla(\Delta f)$.
For a proof of the other formulas, see e.g.\ \cite[pp.~28-29]{Lic61}.
\end{proof}
Observe that on Einstein manifolds, we have the relation $\Delta_L=\Delta_E+2\mu\cdot\identity$ where $\mu$ is the Einstein constant.
 \begin{lem}\label{spectrumdecomposition}Let $(M,g)$ be an Einstein manifold with constant $\mu$. Then the spectrum of $\Delta_E$ on $\Gamma(S^2M)$ can be decomposed as
 \begin{align*}\spectrum(\Delta_E)=\spectrum(\Delta_0-2\mu\cdot\identity)\cup\spectrum_+((\Delta_1-\mu\cdot\identity)|_{W})\cup\spectrum(\Delta_{E}|_{TT}),
 \end{align*}
 where $W=\left\{\alpha\in\Omega^1(M)\mid\delta\alpha=0\right\}$. Here, $\spectrum_+$ denotes the positive part of the spectrum.
 \end{lem}
 \begin{proof}If $(M,g)$ is not the standard sphere\index{sphere}, we consider the decomposition 
  \begin{align*}\Gamma(S^2 M)=C^{\infty}(M)\cdot g\oplus \delta_g^{*}(\Omega^1(M))\oplus TT_g.
   \end{align*}
Let $\left\{v_i\right\}$, ${i\in\N_0}$ be an eigenbasis of $\Delta_0$ to the eigenvalues $\lambda_i^{(0)}$, where $v_0$ is the constant eigenfunction. Let $\left\{\alpha_i\right\}$, ${i\in\N}$, be an eigenbasis of $\Delta_1=\Delta_H-\mu$ acting on $W$ with eigenvalues $\lambda_i^{(1)}$. Let $\left\{h_i\right\}_{i\in \N}$ 
be an eigenbasis of $\Delta_{E}|_{TT}$ with eigenvalues $\lambda^{(2)}_i$. Then $\left\{\nabla v_i\right\}$, ${i\in\N}$, $\left\{\alpha_i\right\}$, ${i\in\N}$ form an eigenbasis of $\Delta_1$ on all $1$-forms\index{1@$1$-form} and
$\left\{v_i\cdot g\right\}$, ${i\in\N_0}$, $\left\{\nabla^2v_i\right\}$, ${i\in\N}$, $\left\{\delta^*\alpha_i\right\}$, $i\in\N$ and $\left\{h_i\right\}$, ${i\in\N}$ form a basis of $\Gamma(S^2M)$.
On the round sphere, we have 
 \begin{align*}\Gamma(S^2M)=[C^{\infty}(M)\cdot g+ \delta_g^{*}(\Omega^1(M))]\oplus \trace_g^{-1}(0)\cap\delta_g^{-1}(0).
\end{align*}
and \begin{align*}C^{\infty}(M)\cdot g\cap \delta_{g}^{*}(\Omega^1(M))=\left\{f\cdot g\in C^{\infty}(M)\cdot g\mid\Delta f=n\cdot f\right\},
\end{align*}
where $n$ is the first nonzero eigenvalue of the Laplacian (see \cite[Lemma 4.57]{Bes08} and \cite[Theorem A]{Ob62}).
If $(M,g)=(S^n,g_{st})$, we therefore have a basis, if we remove from $\left\{\nabla^2v_i\right\}$ the $v_i$ which are the eigenfunctions to the first nonzero eigenvalue of the Laplacian.
By the relation $\Delta_E=\Delta_L-2\mu\cdot\identity$ and Lemma \ref{LL}, we have
\begin{align*}\Delta_E(v_i\cdot g)&=(\lambda_i^{(0)}-2\mu)v_i\cdot g,\\
              \Delta_E(\nabla^2v_i)&=(\lambda_i^{(0)}-2\mu)\nabla^2v_i,\\
              \Delta_E(\delta^*\alpha_i)&=(\lambda_i^{(1)}-\mu)\delta^*\alpha_i,
\end{align*}
which shows that we have obtained a basis of eigentensors of $\Delta_E$.
By Lemma \ref{divfreespectrum} below, $\lambda_i^{(1)}-\mu\geq0$ and equality holds if and only if $\delta^*\omega_i=0$. This finishes the proof of the lemma.
 \end{proof}
 \begin{lem}\label{divfreespectrum}Let $(M,g)$ be an Einstein manifold with constant $\mu$ and $W$ as in Lemma \ref{spectrumdecomposition} above. Then
 \begin{align*} \left\|\nabla\alpha\right\|_{L^2}^2=2\left\|\delta^*\alpha\right\|^2+\mu\left\|\alpha\right\|^2_{L^2}
 \end{align*}
 for any $\alpha\in W$. In particular, $\spectrum((\Delta_1-\mu\cdot\identity)|_{W})$ is nonnegative.
 \end{lem}
 \begin{proof}Let $\left\{e_1,\ldots,e_n\right\}$ be a local orthonormal frame. Then
 \begin{align*}
 \left\|\nabla\alpha\right\|_{L^2}^2&=\int_M \sum_{i,j}(\nabla_{e_i}\alpha(e_j))^2\dv\\
                                              &=\frac{1}{2}\int_M \sum_{i,j}[(\nabla_{e_i}\alpha(e_j)+\nabla_{e_j}\alpha(e_i))^2-2(\nabla_{e_i}\alpha(e_j)\nabla_{e_j}\alpha(e_i))]\dv\\
                                             &=2\left\|\delta^*\alpha\right\|^2+\int_M\sum_{i,j} \alpha(e_j)\nabla^{2}_{e_i,e_j}\alpha(e_i)\dv\\
                                            &=2\left\|\delta^*\alpha\right\|^2+\int_M\sum_{i,j}\alpha(e_j) R_{e_i,e_j}\alpha(e_i)\dv\\
                                              &=2\left\|\delta^*\alpha\right\|^2+\int_M \sum_{j}\alpha(e_j) (\alpha\circ\ric)(e_j)\dv\\
                                              &=2\left\|\delta^*\alpha\right\|^2+\mu\left\|\alpha\right\|^2_{L^2}.
\end{align*}
Here we used $\delta\alpha$ to get from the third line to the fourth.
If $\mu\leq0$, then $\Delta_1-\mu\cdot\identity=\nabla^*\nabla-\mu\cdot\identity$
is obviously nonnegative. By the formula we just have shown, this also holds
if $\mu>0$.
 \end{proof}
 \begin{prop}If $(M,g_1)$ and $(N,g_2)$ are two stable Einstein metrics with $\mu\leq0$, the product manifold $(M\times N,g+ h)$ is also stable.
 \end{prop}
 \begin{proof}By Lemma \ref{spectrumdecomposition} and since $\mu\leq0$, the operators $\Delta_E^M$, $\Delta_E^N$ are nonnegative on all of $\Gamma(S^2M)$ (resp.\ $\Gamma(S^2N)$) if and only if their restriction to $TT$-tensors is, respectively.
 By Proposition \ref{productspectrum}, $\Delta_E^{M\times N}$ is nonnegative since the sum of the spectra\index{spectrum} does not contain negative elements.
 \end{proof}
 If $(M,g)$ and $(N,g_2)$ are stable Einstein manifolds with constant $\mu<0$, it is also quite immediate that 
 \begin{align*}\kernel(\Delta_E^{M\times N}|_{TT})\cong \kernel(\Delta_{E}^{M}|_{TT})\oplus \kernel(\Delta_{E}^{N}|_{TT})
 \end{align*}
 (see \cite[Lemma 3.2]{AMo11}). We show that if $\mu=0$, the situation is slightly more subtle.
 \begin{prop}Let $(M^{n_1},g_1)$ and $(N^{n_2},g_2)$ be stable Ricci-flat\index{Ricci-flat} manifolds. Then
 \begin{align*}\kernel(\Delta_{E}^{M\times N}|_{TT})\cong&\R(n_2\cdot g_1-n_1\cdot g_2)\oplus (\Parallel(\Omega^1(M))\odot\Parallel(\Omega^1(N)))\\
                                                     &\oplus\kernel(\Delta_{E}^{M}|_{TT})\oplus \kernel(\Delta_{E}^{N}|_{TT}).
 \end{align*}
 Here, $\Parallel(\Omega^1(M)),\Parallel(\Omega^1(N))$\index{$\Parallel(M)$, space of parallel $1$-forms on $M$} denote the spaces of parallel $1$-forms\index{1@$1$-form} on $M,N$ respectively.
 If all infinitesimal Einstein deformations of $M$ and $N$ are integrable, then all infinitesimal Einstein deformations of $M\times N$ are integrable.
 \end{prop}
 \begin{proof}By the proof of Proposition \ref{productspectrum}, the kernel\index{kernel} of $\Delta_E^{M\times N}$ is spanned by tensors of the form $v_ik_j$, $w_ih_j$, $\alpha_i\odot\beta_j$
 where $v_i,\alpha_i,h_i$ and $w_i,\beta_i,k_i$ are eigentensors of $\Delta_0,\Delta_1,\Delta_E$ on $M$ and $N$, respectively. By Lemma \ref{spectrumdecomposition}, these operators are nonnegative, so the eigentensors have to lie in the kernel\index{kernel}
 of the corresponding operators. 
 This shows 
 \begin{align*}\kernel(\Delta_{E}^{M\times N})\cong& \R\cdot g_1\oplus\R\cdot g_2\oplus (\Parallel(\Omega^1(M))\odot\Parallel(\Omega^1(N)))\\
                                                   &\oplus\kernel(\Delta_{E}^{M}|_{TT})\oplus \kernel(\Delta_{E}^{N}|_{TT}).
 \end{align*}
 The first assertion follows from restricting $\Delta_{E}^{M\times N}$ to $TT$-tensors.
 Any deformation $h\in\R(n_2\cdot g_1-n_1\cdot g_2)$ is integrable since it can be integrated to a curve of metrics of the form $(g_1)_t+(g_2)_t$ where $(g_1)_t$ and $(g_2)_t$ are just rescalings of $g_1$ and $g_2$.
 This of course does not affect the Ricci-flatness\index{Ricci-flat} of $M\times N$.

 Now, consider the situation where $h\in (\Parallel(\Omega^1(M))\odot\Parallel(\Omega^1(N)))$.
 Let $\alpha_1,\ldots,\alpha_{m_1}$ be a basis of $\Parallel(\Omega^1(M))$ and $\beta_1,\ldots,\beta_{m_2}$ be a basis of $\Parallel(\Omega^1(N))$. Suppose for simplicity that all these 
 forms have constant lengh $1$.
 Then $$h=\sum_{i=1}^{m_1}\sum_{j=1}^{m_2}c_{ij}\cdot\alpha_i\odot\beta_j,$$
 where $c_{ij}\in\R$. We show that $h$ is integrable.
 By the holonomy principle\index{holonomy!principle}, we have parallel decompositions\index{parallel!decomposition}
 \begin{align*}TM=E\oplus\bigoplus_{i=1}^{m_1}(\R\cdot \alpha_i^{\sharp}),\qquad TN=E'\oplus\bigoplus_{j=1}^{m_2}(\R\cdot\beta_j^{\sharp}),
\end{align*}
 and the metrics split as $g_1=\tilde{g}_1+\sum_{i=1}^{m_1}\alpha_i\otimes\alpha_i$, $g_2=\tilde{g}_2+\sum_{j=1}^{m_2}\beta_j\otimes\beta_j$. The metrics $\tilde{g}_1$ and $\tilde{g}_2$ are also Ricci-flat. 
 The tangent bundle of the product\index{product} manifold obviously splits as \index{$\otimes$, tensor product}
 \begin{align*}T(M\times N)= E\oplus E'\oplus\bigoplus_{i=1}^{m_1}(\R\cdot \alpha_i^{\sharp})\oplus\bigoplus_{j=1}^{m_2}(\R\cdot\beta_j^{\sharp}).
 \end{align*}
Observe that $g_1+g_2$ is flat when restricted to
\begin{align*}F=\bigoplus_{i=1}^{m_1}(\R\cdot \alpha_i^{\sharp})\oplus\bigoplus_{j=1}^{m_2}(\R\cdot\beta_j^{\sharp}).
\end{align*}
 Consider the curve of metrics $t\mapsto g_t=g_1+ g_2+th$ on $M\times N$. 
 The metric restricted $E\oplus E'$ does not change and stays flat if we restrict to $F$.
 Thus, $g_t$ is a curve of Ricci-flat metrics, so $h$ is integrable.

If $h\in \kernel(\Delta_{E}^{M}|_{TT})$, then there exists a curve of Einstein metrics $(g_1)_t$ on $M$ tangent
to $h$ by assumption. Consequently, the curve $(g_1)_t\oplus g_2$ is a curve of Einstein metrics on $M\times N$ tangent to $h$, so
$h$ is integrable (considered as an infinitesimal Einstein deformation on $M\times N$).
If $h\in \kernel(\Delta_{E}^{N}|_{TT})$, an analogous argument shows the integrability of $h$.
 \end{proof}
\noindent
Now, let us turn to the case where the Einstein constant is positive. Here, we often use a consequence of a result by Obata \cite{Ob62}: On any Einstein manifold of constant $\mu$, the smallest nonzero eigenvalue of the Laplacian
satisfies $\lambda\geq\frac{n}{n-1}\mu$ and equality holds exactly in the case of the round sphere. We refer to this fact as Obata's eigenvalue estimate.
\begin{lem}\label{kernelindexdecomposition}Let $(M,g)$ be a positive Einstein manifold with constant $\mu$. Then\index{$\ind$, index of a quadratic form}\index{$\mult_{\Delta}(\lambda)$, multiplicity of $\lambda$ as an eigenvalue of $\Delta$}
 \begin{align*}\dimn(\kernel\Delta_E)&=2\cdot\mult_{\Delta_0}(2\mu)+\dimn(\kernel \Delta_{E}|_{TT}),\\
               \ind(\Delta_E)&=1+\mult_{\Delta_0}\left(\frac{n}{n-1}\mu\right)+\sum_{\lambda\in (\frac{n}{n-1}\mu,2\mu)}2\cdot\mult_{\Delta_0}(\lambda)+\ind(\Delta_{E}|_{TT}),
\end{align*}
where $\mult_{\Delta_0}(\lambda)$ is the multiplicity\index{multiplicity} of $\lambda$ as an eigenvalue of $\Delta_0$ and $\ind(\Delta_E)$ is the index of the quadratic form $h\mapsto(\Delta_Eh,h)_{L^2}$.\index{index of a quadratic form}
\end{lem}
\begin{proof}This follows immediately from the proof of Lemma \ref{spectrumdecomposition} and Obata's theorem.
\end{proof}
\begin{prop}\label{productindex}Let $(M^{n_1},g_1)$, $(N^{n_2},g_2)$ be stable Einstein manifolds with constant $\mu>0$. Then
\begin{align*}\dimn(\kernel\Delta_{E}^{M\times N}|_{TT})=&\dimn(\kernel\Delta_{E}^{M}|_{TT})+\dimn(\kernel\Delta_{E}^{N}|_{TT})
                                                          +\mult_{\Delta_0^M}(2\mu)+\mult_{\Delta_0^N}(2\mu),\\
              \ind(\Delta_{E}^{M\times N}|_{TT})=&1+\sum_{\lambda\in (\frac{n_1}{n_1-1}\mu,2\mu)}\mult_{\Delta_0^M}(\lambda)+\sum_{\lambda\in (\frac{n_2}{n_2-1}\mu,2\mu)}\mult_{\Delta_0^N}(\lambda).
\end{align*}
\end{prop}
\begin{proof}We now prove the first assertion. By Lemma \ref{divfreespectrum}, $\Delta_1^M$ and $\Delta_1^N$ are positive. Thus by Proposition \ref{productspectrum}, we have to count the number of eigenvalues (with their multiplicity)
$\lambda_i^{(0)}\in\spectrum(\Delta_0^M)$, $\lambda_i^{(2)}\in\spectrum(\Delta_E^M)$, $\kappa_i^{(0)}\in\spectrum(\Delta_0^N)$, $\kappa_i^{(2)}\in\spectrum(\Delta_E^N)$ such that
$\lambda_i^{(0)}+\kappa_i^{(2)}=0$ and $\lambda_i^{(2)}+\kappa_i^{(0)}=0$.
Consider the first equation. If $\lambda_i^{(0)}=\lambda_0^{(0)}=0$, then also $\kappa_i^{(2)}=0$ and the multiplicity\index{multiplicity} of $\kappa_i^{(2)}$ is given in Lemma \ref{kernelindexdecomposition}. If $\lambda_i^{(0)}>0$,
then $\kappa_i^{(2)}<0$. By Lemma \ref{spectrumdecomposition}, Lemma \ref{divfreespectrum} and since $(M,g_1)$ is stable, $\kappa_i^{(2)}+2\mu=\kappa_i^{(0)}\in\spectrum(\Delta_0^N)$. We thus have to find $\kappa^{(0)}_i$ such that $\lambda_i^{(0)}+\kappa_i^{(0)}=2\mu$ for $\lambda_i^{(0)}>0$.
By Obata's eigenvalue estimate, we have a lower bound $\lambda_i^{(0)},\kappa^{(0)}_i\geq \frac{n}{n-1}\mu$ for nonzero eigenvalues of the Laplacian. Therefore, the only situation which remains possible is that $\lambda_i^{(0)}=2\mu$
and $\kappa_i^{(0)}=\kappa_0^{(0)}=0$. Since eigenvalue zero has always multiplicity $1$, $\kappa_i^{(2)}=\kappa_0^{(0)}-2\mu=-2\mu$ is of multiplicity $1$. Now we do the same game for the equation $\lambda_i^{(2)}+\kappa_i^{(0)}=0$.
We obtain, after summing up both cases,
\begin{align*}\dimn(\kernel\Delta_{E}^{M\times N})=\dimn(\kernel\Delta_{E}^{M}|_{TT})+\dimn(\kernel\Delta_{E}^{N}|_{TT})
                                                          +3\mult_{\Delta_0^M}(2\mu)+3\mult_{\Delta_0^N}(2\mu).
\end{align*}
By the formula
\begin{align}\label{productmultiplicity}\mult_{\Delta_0^{M\times N}}(\tau)=\sum_{\lambda+\kappa=\tau}\mult_{\Delta_0^M}(\lambda)\cdot\mult_{\Delta_0^N}(\kappa)
\end{align}
and by Obata's eigenvalue estimate,
\begin{align*}\mult_{\Delta_0^{M\times N}}(2\mu)=\mult_{\Delta_0^M}(2\mu)+\mult_{\Delta_0^N}(2\mu).
\end{align*}
From Lemma \ref{kernelindexdecomposition}, we get the dimension of $\kernel\Delta_{E}^{M\times N}|_{TT}$.

To show the second assertion, we compute the number of eigenvalues (with multiplicity) satisfiying $\lambda_i^{(0)}+\kappa_i^{(2)}<0$ or $\lambda_i^{(2)}+\kappa_i^{(0)}<0$.
Consider the first inequality. If $\lambda_i^{(0)}=\lambda_0^{(0)}=0$, then $\kappa_i^{(2)}<0$ and the number of such eigenvalues (with multiplicity) is given by Lemma \ref{kernelindexdecomposition}.
If $\lambda_i^{(0)}>0$, then $\lambda_i^{(0)}\geq \frac{n}{n-1}\mu$ and $\kappa_i^{(2)}<-\frac{n}{n-1}\mu$. By Lemma \ref{spectrumdecomposition}, $\kappa_i^{(2)}+2\mu=\kappa_i^{(0)}\in\spectrum(\Delta_0^N)$
and $\kappa_i^{(0)}<\frac{n-2}{n-1}\mu$. By Obata's eigenvalue estimate, $\kappa_i^{(0)}=\kappa_0^{(0)}=0$ and $\kappa_i^{(2)}=-2\mu$ appears with multiplicity $1$. This also implies that $\lambda_i^{(0)}<2\mu$.
\noindent
Similarly, we deal with the inequality $\lambda_i^{(2)}+\kappa_i^{(0)}<0$.
Summing up over both cases, we obtain
\begin{align*}\ind(\Delta_{E}^{M\times N})=  2&+3\sum_{\lambda\in (\frac{n_1}{n_1-1}\mu,2\mu)}\mult_{\Delta_0^M}(\lambda)+3\sum_{\lambda\in (\frac{n_2}{n_2-1},2\mu)}\mult_{\Delta_0^N}(\lambda)\\
                                           &+2\cdot\mult_{\Delta_0^M}\left(\frac{n_1}{n_1-1}\mu\right)+2\cdot\mult_{\Delta_0^N}\left(\frac{n_2}{n_2-1}\mu\right).
\end{align*}
By \eqref{productmultiplicity} and by Obata's eigenvalue estimate,
\begin{align*}\sum_{\lambda\in (0,2\mu)}\mult_{\Delta_0^{M\times N}}(\lambda)=\sum_{\lambda\in (0,2\mu)}\mult_{\Delta_0^M}(\lambda)+\sum_{\lambda\in (0,2\mu)}\mult_{\Delta_0^N}(\lambda)
\end{align*}
and the second assertion follows from Lemma \ref{kernelindexdecomposition}.
\end{proof}
\begin{rem}
 Any product of positive Einstein metrics is unstable. The unstable eigentensor is a traceless linear combination, i.e.\ a volume-perserving deformation which shrinks the one factor of the product and enlarges the other factor.
We also see that Laplacian eigenvalues in the interval $(\frac{n}{n-1}\mu,2\mu)$ enlarge the index\index{index of a quadratic form} of the form $$TT\ni h\mapsto(\Delta_E h,h)_{L^2}.$$
\end{rem}
\begin{rem}Theorem \ref{existenceIED} follows immediately from the above Proposition.
Suppose we have an Eigenfunction $f\in C^{\infty}(M)$ to the Eigenvalue $2\mu$. Then, infinitesimal Einstein deformations on the product space $(M^{n_1}\times N^{n_2},g_1+g_2)$ are constructed as follows:
Consider the linear combination
\begin{align*}
 h=\alpha f\cdot g_1+\beta f\cdot g_2+\gamma\nabla ^2 f\in \Gamma(S^2(M\times N)),
\end{align*}
where $\alpha,\beta,\gamma\in\R$. We have $\Delta_E h=0$ and if $\beta=\frac{2-n_1}{n_2}\alpha$ and $\gamma=\frac{\alpha}{\mu}$, $h\in TT$.
 The nonintegrable infinitesimal Einstein deformations on $S^2\times \CP^{2n}$ mentioned in the introduction are of this form. In general, it is unclear whether such deformations can be integrable.
\end{rem}

\section{Sectional curvature bounds and Einstein deformations}\label{stabilitysectional}
Let $S^2_gM$\index{$S^2_gM$} be the vector bundle
 of symmetric $(0,2)$-tensors whose trace with respect to $g$ vanishes.\index{traceless tensor}
 We define a function $r:M\to \R$ by\index{$r(p)$}
\begin{align}\label{r(p)}r(p)=\sup\left\{\frac{\langle \mathring{R}\eta,\eta\rangle_p}{|\eta|^2_p}\Biggmid\eta\in (S^2_gM)_p\right\}.
\end{align}
 \begin{thm}[{{\cite[Theorem 3.3]{Koi78}}}]\label{koiso}Let $(M,g)$ be an Einstein manifold with Einstein constant $\mu$. If the function $r$ satisfies
\begin{align*}
\sup_{p\in M}r(p)\leq \max\left\{-\mu,\frac{1}{2}\mu\right\},
 \end{align*}
 then $(M,g)$ is stable. If the strict inequality holds, then $(M,g)$ is strictly stable.
 \end{thm}
\begin{proof}For completeness, we sketch the proof of this well-known result.
 Let the two differential operators $D_1$ and $D_2$ be defined\index{differential operator} by\index{$D_1$, a differential operator}\index{$D_2$, a differential operator}
\begin{align*}D_1h(X,Y,Z)=&\frac{1}{\sqrt{3}}(\nabla_X h(Y,Z)+\nabla_Y h(Z,X)+\nabla_Z h(X,Y)),\\
D_2h(X,Y,Z)=&\frac{1}{\sqrt{2}}(\nabla_X h(Y,Z)-\nabla_Y h(Z,X)).
\end{align*}
For the Einstein operator, we have the Bochner formulas\index{Bochner formula}
\begin{align}\label{bochner3}(\Delta_Eh,h)_{L^2}&=\left\|D_1h\right\|_{L^2}^2+2\mu \left\|h\right\|^2_{L^2}-4(\mathring{R}h,h)_{L^2}-2\left\|\delta h\right\|_{L^2}^2,\\
             \label{bochner4} (\Delta_Eh,h)_{L^2}&=\left\|D_2h\right\|_{L^2}^2-\mu \left\|h\right\|^2_{L^2}-(\mathring{R}h,h)_{L^2}+\left\|\delta h\right\|_{L^2}^2,
\end{align}
see \cite[p.~428]{Koi78} or \cite[p.~355]{Bes08} for more details.
Because of the bounds on $r$ and $\delta h=0$, we obtain either $(\Delta_Eh,h)_{L^2}\geq0$ or $(\Delta_Eh,h)_{L^2}>0$ for $TT$-tensors\index{TT@$TT$-tensor} by (\ref{bochner3}) or (\ref{bochner4}).
This proves Theorem \ref{koiso}.
\end{proof}
 \begin{lem}\label{fujitani}Let $(M,g)$ be Einstein and $p\in M$. Let $K_{min}$ and $K_{max}$\index{$K_{min}$, minimal sectional curvature}\index{$K_{max}$ maximal sectional curvature} be the minimum and maximum of its sectional curvature\index{sectional curvature} at $p$, then
 \begin{align}\label{restimate}r(p)\leq\min\left\{(n-2)K_{max}-\mu,\mu-nK_{min}\right\}.
 \end{align}
 If equality holds, i.e.\
 \begin{align}\label{twosidedcurvaturebound}r(p)=(n-2)K_{max}-\mu=\mu-nK_{min},\
 \end{align}
 then $(M,g)$ is even-dimensional. Let $\eta\in (S^2_gM)_p$ be such that $\mathring{R}\eta=r(p) \eta$. Then $\eta$ has only two eigenvalues $\lambda,-\lambda$
 and the eigenspaces\index{eigenspace} $E(\lambda),E(-\lambda)$ are both of dimension $m=n/2$. Moreover, $K(P)=K_{max}$ for each plane $P$ lying in either $E(\lambda)$ or $E(-\lambda)$ and $K(P)=K_{min}$ if $P$ is spanned
 by one vector in $E(\lambda)$ and one in $E(-\lambda)$.
 \end{lem}
 \begin{proof}Estimate \eqref{restimate} was already proven in \cite{Fuj79}. We redo the proof of \cite[Lemma 12.71]{Bes08}.
Choose $\eta$ such that $\mathring{R}\eta=r(p)\eta$. Let $\left\{e_1,\ldots,e_n\right\}$ be an orthonormal basis in which $\eta$ is
 diagonal with eigenvalues $\lambda_1,\ldots,\lambda_n$ such that $\lambda_1=\sup|\lambda_i|$ and $\sum \lambda_i=0$. Then
 \begin{align*}r(p)\lambda_1=(\mathring{R}\eta)(e_1,e_1)=\sum_i K_{i1}\lambda_i,
\end{align*}
 where $K_{i1}$ is the sectional curvature\index{sectional curvature} of the plane\index{plane} spanned by $e_i$ and $e_1$.\index{$K$, sectional curvature}
 Thus,
 \begin{equation}\begin{split}\label{fuji1}r(p)\lambda_1&=\sum_{i\neq1}K_{max}\lambda_i-\sum_{i\neq1}(K_{max}-K_{i1})\lambda_i\\
                         &\leq-\lambda_1 K_{max}+\lambda_1\sum_{i\neq1}(K_{max}-K_{i1})\\
                         &=((n-2)K_{max}-\mu)\lambda_1.
 \end{split}\end{equation}
 On the other hand,
 \begin{equation}\begin{split}\label{fuji2}r(p)\lambda_1&=\sum_{i\neq1}K_{min}\lambda_i+\sum_{i\neq1}(K_{i1}-K_{min})\lambda_i\\
                        &\leq-\lambda_1 K_{min}+\lambda_1\sum_{i\neq1}(K_{i1}-K_{min})\\
                        &=(-nK_{min}+\mu)\lambda,
 \end{split}
\end{equation}
so we have proven the first assertion. Suppose now that \eqref{twosidedcurvaturebound} holds, then equality must hold both in \eqref{fuji1} and \eqref{fuji2}.
 From (\ref{fuji1}), we get that either $\lambda_i=-\lambda_1$ or $K_{i1}=K_{max}$ whereas (\ref{fuji2}) implies $\lambda_i=\lambda_1$ or
 $K_{i1}=K_{min}$ for each $i$. Thus there only exist two eigenvalues $\lambda$ and $-\lambda$ which are of same multiplicity since the trace of $\eta$ vanishes. In particular, $(M,g)$ is even-dimensional.

 Let $P\subset T_pM$ be a plane which satisfies one of the assumptions of the lemma. We then may assume that $P$ is spanned by two vectors of the eigenbasis we have chosen. If $P\subset E(\lambda)$ or $P$ is spanned by two
 vectors in $E(\lambda)$, $E(-\lambda)$, respectively, we may assume $e_1\in P$. Then the assertions follow from the above. If $P\subset E(-\lambda)$, we may replace $\eta$ by $-\eta$ and the roles of $E(\lambda)$ and $E(-\lambda)$ interchange.
 \end{proof}
\noindent
From Theorem \ref{koiso} and the first part of Lemma \ref{fujitani}, the Corollaries \ref{pinching} and \ref{stabilitywhenK<0} are consequences.
\begin{rem}
 Observe that we also get stability by the above, if we just assume a one-sided bound on the sectional curvature in terms of the Einstein constant.
\end{rem}

 \begin{prop}\label{splittingtheorem1}Let $(M,g)$ be an Einstein manifold with constant $\mu$ such that the sectional curvature lies in the interval $[(n-2)/3n,1]\cdot K_{max}$. Then $(M,g)$ is stable.
 If $\kernel{\Delta_{E}|_{TT}}$ is nontrivial, $M^n$ is even-dimensional. Furthermore, there exists an orthogonal splitting\index{orthogonal splitting} $TM=\mathcal{E}\oplus\mathcal{F}$
 into two subbundles\index{subbundle} of dimension $n/2$. The two $C^{\infty}(M)$-bilinear maps\index{bilinear}
 \begin{align*}I\colon\Gamma(\mathcal{E})\times\Gamma(\mathcal{E})&\to \Gamma(\mathcal{F}),\qquad(X,Y)\mapsto pr_{\mathcal{F}}(\nabla_X Y)
\end{align*}
 and 
 \begin{align*}II\colon\Gamma(\mathcal{F})\times\Gamma(\mathcal{F})&\to \Gamma(\mathcal{E}),\qquad(X,Y)\mapsto pr_{\mathcal{E}}(\nabla_X Y)
 \end{align*}
 are both antisymmetric in $X$ and $Y$. Moreover, the sectional curvature\index{sectional curvature} of a plane\index{plane} $P$ is equal to $K_{max}$ if $P$ either lies in $\mathcal{E}$ or $\mathcal{F}$.
 If $P=\mathrm{span}\{e,f\}$ with $e\in\mathcal{E}$ and $f\in\mathcal{F}$, then $K(P)=K_{min}$.
 \end{prop}
 \begin{proof}Because of the curvature assumpions, $\mu\geq \frac{2}{3}(n-2)K_{max}$ or $\mu\leq 2nK_{min}$ at each point. In both cases,
 the function $r$ from Lemma \ref{fujitani} satisfies $r\leq\frac{1}{2}\mu$. Thus, $r_0\leq \frac{1}{2}\mu$
 and Theorem \ref{koiso} implies that $(M,g)$ is stable. Suppose now there exists $h\in\kernel\Delta_{E}|_{TT}$, $h\neq0$. Then by (\ref{bochner4}),
 \begin{align*}0=(\Delta_{E}h,h)_{L^2}&=\left\|D_1 h\right\|^2_{L^2}+2\mu \left\| h\right\|_{L^2}^2-4(h,\mathring{R}h)_{L^2}
                                       \geq0+2\mu \left\| h\right\|_{L^2}^2-2\mu \left\| h\right\|_{L^2}^2=0.
 \end{align*}
 Therefore, $D_1h\equiv0$ and $\langle\mathring{R}h,h\rangle_p\equiv \frac{\mu}{2}|h|_p^2$ for all $p\in M$. The second equality implies that
 \begin{align*}\mu=\frac{2}{3}(n-2)K_{max}=2nK_{min}
 \end{align*} and
 \begin{align*}r(p)=(n-2)K_{max}-\mu=\mu-nK_{min}.
 \end{align*} Thus, Lemma \ref{fujitani} applies and at each point where
 $h\neq0$, the tangent space splits into the two eigenspaces\index{eigenspace} of $h$, i.e.\ $T_pM=E_p(\lambda)\oplus E_p(-\lambda)$. 
Since $D_1 h\equiv0$, we have
 \begin{align}\label{symmcovder}g(\nabla_{e_i}h(e_j),e_k)+g(\nabla_{e_j}h(e_k),e_i)+g(\nabla_{e_k}h(e_i),e_j)=0
 \end{align} 
for any local orthonormal frame $\left\{e_1,\ldots,e_n\right\}$. Here, we considered $h$ as an endomorphism $h:TM\to TM$.
Choose a local eigenframe\index{eigenframe} of $h$ around some $p$ outside the zero set of $h$.
A straightforward calculation shows
 \begin{align}\label{covder}\langle \nabla_{e_i}h(e_j),e_k\rangle&=(\nabla_{e_i}\lambda_j)\delta_{jk}+\lambda_j \Gamma_{ij}^k-\lambda_k\Gamma_{ij}^k,
 \end{align}
 where $\lambda_j$ is the eigenvalue of $e_j$.
 Now we rewrite (\ref{symmcovder}) as
 \begin{equation}\begin{split}\label{symmcovder2}(\lambda_j-\lambda_k)\Gamma_{ij}^k&+(\lambda_k-\lambda_i)\Gamma_{jk}^i+(\lambda_i-\lambda_j)\Gamma_{ki}^j%\\&
 =-(\nabla_{e_i}\lambda_j)\delta_{jk}-(\nabla_{e_j}\lambda_k)\delta_{ki}-(\nabla_{e_k}\lambda_i)\delta_{ij}.
 \end{split}\end{equation}
 If we choose $i=j=k$, we obtain
 \begin{align*}0=-3(\nabla_{e_i}\lambda_i).
 \end{align*}
 Since $\lambda_i=\pm \lambda$, it is immediate that $\lambda$ is constant and it is nonzero. Thus, we obtain a global splitting
 $TM=\mathcal{E}\oplus\mathcal{F}$
 where the two distributions\index{distribution!of the tangent bundle} are defined by
 \begin{align*}\mathcal{E}=\bigcup_{p\in M}E_p(\lambda),\qquad \mathcal{F}=\bigcup_{p\in M}E_p(-\lambda).
 \end{align*}
 By Lemma $\ref{fujitani}$, the assertion about the sectional curvatures is immediate. To finish the proof, it just remains to show the antisymmetry of the maps $I,II$, respectively.
 
Let $\left\{e_1,\ldots,e_n\right\}$ be the eigenframe\index{eigenframe} from before.
 We may assume that $e_1,\ldots,e_{n/2}$ are local sections in $\mathcal{E}$ and $e_{{n/2}+1},\ldots,e_n$ are local sections in $\mathcal{F}$.
 Choose $i,j\in\left\{1,\ldots,{n/2}\right\}$, $k\in\left\{{n/2}+1,\ldots,n\right\}$. Then
 $\lambda_i=\lambda_j=\lambda$, $\lambda_k=-\lambda$ and (\ref{symmcovder2}) yields
 \begin{align}\label{symmcovder3}0=2\lambda(\Gamma_{ij}^k+\Gamma_{ji}^k),
 \end{align}
 since the right-hand side of (\ref{symmcovder2}) vanishes for any $i,j,k$.
 Now consider the map $I$. We have\index{$\mathrm{pr}$, projection map}
 \begin{align}\label{christoffel}I(e_i,e_j)=\mathrm{pr}_{\mathcal{F}}(\nabla_{e_i}e_j)=\sum_{k=n/2+1}^{n} \Gamma_{ij}^ke_k,
 \end{align}
 and by (\ref{symmcovder3}), we immediately get $I(e_i,e_j)=-I(e_j,e_i)$. Similarly, antisymmetry is shown for $II$.
 \end{proof}
\noindent
 Now let us turn to the case of nonpositive secional curvature.
  \begin{defn}Let $(M,g)$ be a Riemannian manifold and let $\left\{e_1,\ldots,e_n\right\}$ be an orthonormal frame at $p\in M$.
 Then $K_{ij}=R_{ijji}$ is the sectional curvature\index{sectional curvature} of the plane spanned by $e_i$ and $e_j$ if $i\neq j$ and is zero if $i=j$.
 We count the number of $j$ such that $K_{i_0 j}=0$ for a given $i_0$ and call the maximum of such numbers
 over all orthonormal frames at $p$ the flat dimension of $M$ at $p$, denoted by $\text{fd}(M)_p$.\index{$\text{fd}(M)_p$, flat dimension of $M$ at $p$}
 The number $\text{fd}(M)=\sup_{p\in M}\text{fd}(M)_p$ is called the flat dimension of $M$.\index{flat dimension}\index{$\text{fd}(M)$, flat dimension of $M$}
 \end{defn}
 \begin{prop}[{{\cite[Proposition 3.4]{Koi78}}}]
 Let $(M,g)$ be a non-flat Einstein manifold with nonpositive sectional curvature. Then $(M,g)$ is stable. 
 If $\kernel (\Delta_{E}|_{TT})$ is nontrivial, the flat dimension of $M$ satisfies\index{flat dimension}
 $\text{fd}(M)_p\geq \lceil \frac{n}{2}\rceil$ at each $p\in M$.
 \end{prop}
 If in addition, a lower bound on the sectional curvature is assumed, we obtain stronger consequences of the existence of infinitesimal Einstein deformations:
 \begin{prop}\label{splittingtheorem2}Let $(M,g)$ a non-flat Einstein manifold with nonpositive sectional curvature\index{sectional curvature} and Einstein constant $\mu$. If $K_{min}>\frac{2}{n}\mu$, then $(M,g)$ is strictly stable. If $K_{min}\geq \frac{2}{n}\mu$, 
 then $(M,g)$ is stable.
 If $\kernel{\Delta_{E}|_{TT}}$ is nontrivial, then $M$ is even-dimensional and we have an orthogonal splitting\index{orthogonal splitting}  $TM=\mathcal{E}\oplus\mathcal{F}$. Both subbundles\index{subbundle} are of dimension $n/2$.
 The $C^{\infty}(M)$-bilinear maps\index{bilinear}
 \begin{align*}I\colon\Gamma(\mathcal{E})\times\Gamma(\mathcal{E})&\to \Gamma(\mathcal{F}),\qquad(X,Y)\mapsto pr_{\mathcal{F}}(\nabla_X Y)
 \end{align*}
 and 
 \begin{align*}II\colon\Gamma(\mathcal{F})\times\Gamma(\mathcal{F})&\to \Gamma(\mathcal{E}),\qquad(X,Y)\mapsto pr_{\mathcal{E}}(\nabla_X Y)
 \end{align*}
 are symmetric. Moreover, $K(P)=0$ for any plane\index{plane} lying in $\mathcal{E}$ or $\mathcal{F}$.
 \end{prop}
 \begin{proof}
 Since the sectional curvature\index{sectional curvature} is nonpositive but not identically zero, the Einstein constant is negative.
 Now we follow the same strategy as in the proof of Proposition \ref{splittingtheorem1}.
 If $K_{min}>\frac{2}{n}\mu$, then $r_p<-\mu$ and by Proposition \ref{koiso}, $(M,g)$ is strictly stable.
 If $K_{min}\geq\frac{2}{n}\mu$ and $h\in \kernel(\Delta_{E}|_{TT})$, we obtain from (\ref{bochner4}) that
 \begin{align*}0=(\Delta_Eh,h)_{L^2}&=\left\|D_2 h\right\|^2_{L^2}-\mu \left\| h\right\|_{L^2}^2-(h,\mathring{R}h)_{L^2}
                                       \geq-\mu\left\| h\right\|_{L^2}^2+\mu\left\| h\right\|_{L^2}^2=0.
 \end{align*}
 Consequently, $D_2 h\equiv0$ and $r(p)=K_{max}-\mu=\mu-nK_{min}$. Again by Lemma \ref{fujitani}, there is a splitting $T_p M=E_p(\lambda)\oplus E_p(-\lambda)$ at each point $p\in M$ where $h\neq0$
  and $E_p(\pm\lambda)$ is the $n/2$-dimensional eigenspaces\index{eigenspace} of $h$ to the eigenvalue $\pm\lambda$, respectively.
 Evidently, $(M,g)$ is even-dimensional.
 We will now show that $\lambda$ is constant in $p$.
 Let $\left\{e_1,\ldots ,e_{n}\right\}$ be a local eigenframe of $h$ such that $e_1,\ldots ,e_{n/2}\in E(\lambda)$ and $e_{{n/2}+1},\ldots, e_{n}\in E(-\lambda)$ and let $\lambda_1\equiv\ldots\equiv\lambda_{n/2}$ and $\lambda_{n/2+1}\equiv\ldots\equiv\lambda_n$
be the corresponding eigenfunctions.
 Since $D_2 h\equiv0$, \eqref{covder} yields
 \begin{align}\label{antisymmcovder}(\lambda_j-\lambda_k)\Gamma_{ij}^k-(\lambda_i-\lambda_k)\Gamma_{ij}^k=-(\nabla_{e_i}\lambda_j)\delta_{jk}+(\nabla_{e_j}\lambda_i)\delta_{ik}
 \end{align}
 for $1\leq i,j,k\leq n$. Choose $i\neq j$ and $j=k$ such that $e_i,e_j,e_k$ lie in the same eigenspace.\index{eigenspace} Then by (\ref{antisymmcovder}),
 \begin{align*}0=-\nabla_{e_i}\lambda_j
 \end{align*}
 and since $\lambda_j$ equals either $\lambda$ or $-\lambda$, the eigenvalues of $h$ are constant in $p$.
 A splitting of the tangent bundle is obtained by $TM=\mathcal{E}\oplus\mathcal{F}$
 where the two distributions\index{distribution!of the tangent bundle} are defined by
 \begin{align*}\mathcal{E}=\bigcup_{p\in M}E_p(\lambda),\qquad \mathcal{F}=\bigcup_{p\in M}E_p(-\lambda).
 \end{align*}
 The flatness of planes\index{plane} in $\mathcal{E}$ and $\mathcal{F}$ follows from Lemma \ref{fujitani}. It remains to show the symmetry of $I$ and $II$.
 Let $\left\{e_1,\ldots e_n\right\}$ an orthonormal frame such that $e_1,\ldots,e_{n/2}$ are local sections in $\mathcal{E}$ and 
$e_{n/2+1},\ldots,e_n$ are local sections in $\mathcal{F}$.
 Let $i,j\in\left\{1,\ldots n/2\right\}$ and $k\in\left\{n/2+1,\ldots,n\right\}$.
 By (\ref{antisymmcovder}),
 \begin{align*}2\lambda\Gamma_{ij}^k-2\lambda\Gamma_{ji}^k=0.
 \end{align*}
 and since
 \begin{align}I(e_i,e_j)=\sum_{k=n/2+1}^{n} \Gamma_{ij}^ke_k,
 \end{align}
$I$ is symmetric. The symmetry of $II$ is shown by the same arguments. It is furthermore easy to see that both maps are $C^{\infty}(M)$-bilinear\index{bilinear}.
 \end{proof}
 \begin{rem}By symmetry of the operators $I$ and $II$, the map $(X,Y)\mapsto [X,Y]$ preserves the splitting $TM=\mathcal{E}\oplus\mathcal{F}$. Thus,
 both distributions\index{distribution!of the tangent bundle} are integrable by the Frobenius theorem\index{Frobenius theorem}. 
 \end{rem}

\end{document}